\documentclass[10pt,leqno]{article}

\usepackage{amsmath,amssymb,amsthm,mathrsfs,dsfont,bm}

\usepackage[margin=3cm]{geometry} 

\usepackage{titlesec,hyperref}

\usepackage{color}

\usepackage{fancyhdr}
\titleformat{\subsection}{\it}{\thesubsection.\enspace}{1pt}{}

\newtheorem{theo}{Theorem}[section]
\newtheorem{lemm}[theo]{Lemma}
\newtheorem{defi}[theo]{Definition}

\newtheorem{prop}[theo]{Proposition}

\numberwithin{equation}{section}

\allowdisplaybreaks 

\begin{document}
	\title{Blow-up phenomena and the local well-posedness of a Two-Component b-Family Equations in Besov spaces}
	\author{
		Lingli $\mbox{Hu}^{1,2}$ \footnote{email: 847285225@qq.com} \quad and\quad
		Zhaoyang $\mbox{Yin}^{3,1}$ \footnote{email: mcsyzy@mail.sysu.edu.cn}\\
		\quad 
		$^1\mbox{Department}$ of Mathematics, Sun Yat-sen University,\\
		Guangzhou, 510275, China\\
		$^2\mbox{College}$ of Mathematics and Statistics,
		Jishou University,\\ Jishou, 416000, China\\
        $^3\mbox{School}$ of Science, Shenzhen Campus of Sun Yat-sen University,\\ Shenzhen, 518107, China
	}
	\date{}
	\maketitle
	\begin{abstract}
		In this paper, we first establish the local well-posednesss of a Two-Component b-Family equations in nonhomogeneous Besov spaces $B^{1+\frac 1 p}_{p,1}$ with $1\leq p<+\infty.$ Then we present a new blow-up result for the Two-Component b-Family equations.
	\end{abstract}
	
	
	\tableofcontents
	
	\section{Introduction}
	\par

	In this paper,  we consider the Cauchy problem of a Two-Component b-Family equations 
	\begin{equation}\label{eq0}
		\left\{\begin{array}{l}
		\rho_t-k_3(\rho u)_x=0,\quad t>0,\ x\in\mathbb{R},  \\
		m_t-um_x=k_1mu_x+k_2\rho\rho_x,\\
		(\rho,m)|_{t=0}=(\rho_0(x),m_0(x)), \quad x\in\mathbb{R}.
		\end{array}\right.
	\end{equation}
The system \eqref{eq0} was recently introduced by Guha in \cite{a2}. Equation \eqref{eq0} has two cases:  ${\rm(i)}~k_1=b, k_2=2b, k_3= 1$,  ${\rm(ii)}~k_1=b+1, k_2=2, k_3=b$. Furthermore, the constant $k_2>0(k_2<0)$ represents the case where the acceleration points upwards (downwards), particularly, $k_2=1~(or~k_2=-1~)$ represents the case where the acceleration of gravity points upwards(downwards) \cite{a3}. $k_3$ denotes the mass conservation ratio of the solution. The two-component Camassa–Holm system and the two-component Degasperis–Procesi system can be viewed as two members of the system \eqref{eq0}.
For $\rho \equiv 0$, the system \eqref{eq0} becomes the b-family equation
\begin{align}\label{y1.1}
	(1-\alpha^2\partial_{xx})u_t+c_0u_x+(b+1)uu_x+\lambda u_{xxx}=\alpha^2(bu_xu_{xx}+uu_{xxx}).
\end{align}
 Eq.\eqref{y1.1} can be derived as the family of asymptotically equivalent shallow water wave equations that emerge at quadratic order accuracy for any $b\neq-1$ by an appropriate Kodama transformation; cf. \cite{a4,a5}. For the case $b=-1$, the corresponding Kodama transformation is singular and the asymptotic ordering is violated; cf.\cite{a4,a5}.

With $\alpha=0$ and $b= 2$ in Eq. \eqref{y1.1}, we find the well-known KdV equation which describes the unidirectional propagation of waves at the free surface of shallow water under the influence of gravity \cite{a6}. The Cauchy problem of the KdV equation has been studied by many authors \cite{a7,a8,a9} and a satisfactory local or global (in time) existence theory is now available \cite{a8,a9}. For  $b=2$ and $\lambda=0$, Eq. \eqref{y1.1} becomes the Camassa–Holm equation, modelling the unidirectional propagation of shallow water waves over a flat bottom. The Cauchy problem of the Camassa–Holm equation has been the subject of a number of studies, for example \cite{a10,a11}. For $b=3$ and $c_0=\lambda=0$, we find the Degasperis-Procesi equation \cite{a12} from Eq.\eqref{y1.1}, which is regarded as a model for nonlinear shallow water dynamics. There are also many papers involving the Degasperis-Procesi equation,
cf.\cite{a13,yin2003}. The advantage of the Camassa–Holm equation and the Degasperis-Procesi equation in comparison with the KdV equation lies in the fact that these two equations have peakon solitons and model wave breaking\cite{a14,a15}.
In \cite{a17,a18}, the authors studied Eq.\eqref{y1.1} on the line and on the circle, respectively for $\alpha>0$ and $b, c_0, \lambda\in\mathbb{R}$. They \cite{a17,a18} established the local well-posedness, described the precise blow-up scenario, proved that the equation has strong solutions which exist globally in time and blow up in finite time. Moreover, the authors showed the existence of global weak solution to Eq.\eqref{y1.1} on the line and on the circle, respectively.
For $\rho\not\equiv0$, if  $k_1=2$, the system \eqref{eq0} becomes the two-component Camassa–Holm system 
\begin{equation}\label{eq0.1}
	\left\{\begin{array}{l}
		\rho_t+(\rho u)_x=0,\quad t>0,\ x\in\mathbb{R},  \\
		m_t+um_x=-2mu_x-\theta\rho\rho_x,
	\end{array}\right.
\end{equation}
where $m=u-u_{xx}$, $\theta=\pm 1$
 was derived by Constantin and Ivanov \cite{a19} in the context of shallow water theory. The variable $u(t,x)$ describes the horizontal velocity of the fluid and the variable $\rho (t,x)$ is in connection with the horizontal deviation of the surface from equilibrium, all measured in dimensionless units \cite{a19}. They also pointed out that the global solution of the equations exists under some initial conditions, while the wave breaking phenomenon occurs under other conditions. After that, many research results are devoted to the two-component CH equations \cite{a20,a21}. Significantly, the case $k_1=3$ is completely integrable and its inverse scattering and peakons were studied in \cite{a21}. For $\rho\equiv 0$ and $k_1= 3$, Eq.\eqref{eq0} turns into the two-component DP equations.
 In \cite{a22}, Jin and Guo investigated some aspects of blow up mechanism, traveling wave solutions and the persistence properties of the two-component DP equations. In \cite{a23}, for $\rho\not\equiv 0$ and $b\in\mathbb{R}$, Du and Wu investigated the well-posedness of Eq.\eqref{eq0} in $H^{s-1,p}(\mathbb{R})\times H^{s,p}(\mathbb{R}), s>\max\{2,\frac{3}{2}+\frac{1}{p}\}, p\in(1,\infty)$ for $p\neq2$.

	Our purpose is to adopt the method in \cite{yeweikui2023} and give a new result for the local well-posedness of the Two-Component b-Family equations in $B^{1+\frac 1 p}_{p,1}\times B^{\frac 1 p}_{p,1}$ with $p\in[1,+\infty).$  The main method is to use compactness theory and lagrangian coordinate transformation to overcome the existence and uniqueness of the solution in $B^{1+\frac 1 p}_{p,1}\times B^{\frac 1 p}_{p,1}$ with $p\in[1,+\infty).$ And we also give blow-up results.
	
	This paper is organized as follows. In Section 2, we give some preliminaries which will be used in sequel and the local well-posedness of \eqref{eq1} in  $B^{s}_{p,r}\times B^{s-1}_{p,r}$ with $ s>\max\{\frac 3 2, 1+\frac 1 p\}$  and  $B^{\frac 3 2}_{2,1}\times B^{\frac 1 2}_{2,1}$. In Section 3, we prove the local well-posedness of \eqref{eq1} in $B^{1+\frac{1}{p}}_{p,1}\times B^{\frac 1 p}_{p,1}$ with $p\in [1,+\infty)$ in the sence of the Hadamard (i.e. \eqref{eq1}has a local solution in $B^{1+\frac{1}{p}}_{p,1}\times B^{\frac 1 p}_{p,1}, p\in [1,+\infty)$ and the solution is unique when $k_3=1$. Moreover,  the solution depends continuously on the initial data). The main approachs are based on the Littlewood-Paley theory and transport equations theorey. In Section 4, we present  a blow-up results of \eqref{eq1}.
	\section{Preliminaries}
	\par
	In this section, we first review some basic properties on the Littlewood-Paley theory, which can be found in \cite{chemin2011}.

	Let $\chi$ and $\phi$ be a radical, smooth, and valued in the interval [0,1], belonging respectively to $\mathcal{D}(\textbf{B})$ and $\mathcal{D}(\textbf{C})$,  where $\textbf{B}=\{\xi\in\mathbb{R}^d:|\xi|\leq\dfrac{4}{3}\},$  $\textbf{C}=\{\xi\in\mathbb{R}^d:\dfrac{4}{3}\leq|\xi|\leq\dfrac{8}{3}\}.$ Denote $\mathcal{F}$ by the Fourier transform and $\mathcal{F}^{-1}$ by its inverse. For any $u\in\mathcal{S}'(\mathbb{R}^{d})$ and $j\in\mathbb{Z}$, define $\Delta_{j}u=0$ for $j\leq-2$; $\Delta_{-1}u=\mathcal{F}^{-1}(\chi\mathcal{F}u)$;  $\Delta_{j}u=\mathcal{F}^{-1}(\phi(2^{-j}\cdot)\mathcal{F}u)$ for $j\geq 0$; and $S_{j}u=\sum\limits_{j'<j}\Delta_{j'}u$.
	
	Let $s\in\mathbb{R}$, $1\leq p,r\leq\infty$. The nonhomogeneous Besov space $B^s_{p,r}(\mathbb{R})$ is defined by
	$$B^s_{p,r}=B^s_{p,r}(\mathbb{R}^{d})=\{u\in\mathcal{S}'(\mathbb{R}^{d}):\|u\|_{B^s_{p,r}}=\bigg\|(2^{js}\|\Delta_{j}u\|_{L^{P}})_{j}\bigg\|_{l^{r}(\mathbb{Z})}<\infty\}.$$ 
	
	The nonhomogeneous Sobolev space is defined by
	$$H^{s}=H^{s}(\mathbb{R}^{d})=\{u\in\mathcal{S}'(\mathbb{R}^{d}):u\in L_{loc}^{2}(\mathbb{R}^{d}), \|u\|_{H^{s}}^2=\int\limits_{\mathbb{R}^d}(1+|\xi|^2)^s|\mathcal{F}u(\xi)|^2d\xi<\infty\}.$$
	
	The nonhomogeneous Bony's decomposition is defined by $uv=T_{u}v+T_{v}u+R(u,v)$ with
	$$T_{u}v=\sum\limits_{j}S_{j-1}u\Delta_{j}v,~ R(u,v)=\sum\limits_{j}\sum\limits_{|j'-j|\leq1}\Delta_{j}u\Delta_{j'}v.$$
	
	Next, we introduce some properties about Besov spaces. 
	\begin{prop}\cite{chemin2011,he2017}
		Let $s\in\mathbb{R},\ 1\leq p,p_1,p_2,r,r_1,r_2\leq\infty.$  \\
		(1) $B^s_{p,r}$ is a Banach space, and is continuously embedded in  $\mathcal{S}'$. \\
		(2) If $r<\infty$, then $\lim\limits_{j\rightarrow\infty}\|S_j u-u\|_{B^s_{p,r}}=0$. If $p,r<\infty$, then $C_0^{\infty}$ is dense in $B^s_{p,r}$. \\
		(3) If $p_1\leq p_2$ and $r_1\leq r_2$, then $ B^s_{p_1,r_1}\hookrightarrow B^{s-d(\frac 1 {p_1}-\frac 1 {p_2})}_{p_2,r_2}. $
		If $s_1<s_2$, then the embedding $B^{s_2}_{p,r_2}\hookrightarrow B^{s_1}_{p,r_1}$ is locally compact. \\
		(4) $B^s_{p,r}\hookrightarrow L^{\infty} \Leftrightarrow s>\frac d p\ \text{or}\ s=\frac d p,\ r=1.
		\quad $ \\
		(5) Fatou property: if $(u_n)_{n\in\mathbb{N}}$ is a bounded sequence in $B^s_{p,r}$, then an element $u\in B^s_{p,r}$ and a subsequence $(u_{n_k})_{k\in\mathbb{N}}$ exist such that
		$$ \lim_{k\rightarrow\infty}u_{n_k}=u\ \text{in}\ \mathcal{S}'\quad \text{and}\quad \|u\|_{B^s_{p,r}}\leq C\liminf_{k\rightarrow\infty}\|u_{n_k}\|_{B^s_{p,r}}. $$
		(6) Let $m\in\mathbb{R}$ and $f$ be a $S^m$-mutiplier, (i.e. f is a smooth function and satisfies that $\forall\alpha\in\mathbb{N}^d$, $\exists C=C(\alpha)$, such that $|\partial^{\alpha}f(\xi)|\leq C(1+|\xi|)^{m-|\alpha|},\ \forall\xi\in\mathbb{R}^d)$.
		Then the operator $f(D)=\mathcal{F}^{-1}(f\mathcal{F}\cdot)$ is continuous from $B^s_{p,r}$ to $B^{s-m}_{p,r}$.
	\end{prop}
	We give two useful interpolation inequalities.
	\begin{prop}\label{prop}\cite{chemin2011,he2017}
		(1) If $s_1<s_2$, $\theta \in (0,1)$, and $(p,r)$ is in $[1,\infty]^2$, then we have
		$$ \|u\|_{B^{\theta s_1+(1-\theta)s_2}_{p,r}}\leq \|u\|_{B^{s_1}_{p,r}}^{\theta}\|u\|_{B^{s_2}_{p,r}}^{1-\theta}. $$
		(2) If $s\in\mathbb{R},\ 1\leq p\leq\infty,\ \varepsilon>0$, a constant $C=C(\varepsilon)$ exists such that
		$$ \|u\|_{B^s_{p,1}}\leq C\|u\|_{B^s_{p,\infty}}\ln\Big(e+\frac {\|u\|_{B^{s+\varepsilon}_{p,\infty}}}{\|u\|_{B^s_{p,\infty}}}\Big). $$
	\end{prop}
	\begin{prop}\cite{chemin2011}
		Let $s\in\mathbb{R},\ 1\leq p,r\leq\infty.$
		\begin{equation*}\left\{
			\begin{array}{l}
				B^s_{p,r}\times B^{-s}_{p',r'}\longrightarrow\mathbb{R},  \\
				(u,\phi)\longmapsto \sum\limits_{|j-j'|\leq 1}\langle \Delta_j u,\Delta_{j'}\phi\rangle,
			\end{array}\right.
		\end{equation*}
		defines a continuous bilinear functional on $B^s_{p,r}\times B^{-s}_{p',r'}$. Denote by $Q^{-s}_{p',r'}$ the set of functions $\phi$ in $\mathcal{S}'$ such that
		$\|\phi\|_{B^{-s}_{p',r'}}\leq 1$. If $u$ is in $\mathcal{S}'$, then we have
		$$\|u\|_{B^s_{p,r}}\leq C\sup_{\phi\in Q^{-s}_{p',r'}}\langle u,\phi\rangle.$$
	\end{prop}
	We then have the following product laws:
	\begin{lemm}\label{product}\cite{chemin2011,he2017}
		(1) For any $s>0$ and any $(p,r)$ in $[1,\infty]^2$, the space $L^{\infty} \cap B^s_{p,r}$ is an algebra, and a constant $C=C(s,d)$ exists such that
		$$ \|uv\|_{B^s_{p,r}}\leq C(\|u\|_{L^{\infty}}\|v\|_{B^s_{p,r}}+\|u\|_{B^s_{p,r}}\|v\|_{L^{\infty}}). $$
		(2) If $1\leq p,r\leq \infty,\ s_1\leq s_2,\ s_2>\frac{d}{p} (s_2 \geq \frac{d}{p}\ \text{if}\ r=1)$ and $s_1+s_2>\max(0, \frac{2d}{p}-d)$, there exists $C=C(s_1,s_2,p,r,d)$ such that
		$$ \|uv\|_{B^{s_1}_{p,r}}\leq C\|u\|_{B^{s_1}_{p,r}}\|v\|_{B^{s_2}_{p,r}}. $$
		(3) If $1\leq p\leq 2$,  there exists $C=C(p,d)$ such that
		$$ \|uv\|_{B^{\frac d p-d}_{p,\infty}}\leq C \|u\|_{B^{\frac d p-d}_{p,\infty}}\|v\|_{B^{\frac d p}_{p,1}}. $$
	\end{lemm}
	The Gronwall lemma as follows.
	\begin{lemm}\label{osgood}\cite{chemin2011}
		Let $f(t), g(t)\in C^{1}([0,T]), f(t), g(t)>0.$ Let $h(t)$ is a continuous function on $[0,T].$ Assume that, for any $t\in [0,T]$ such that
		$$\frac 1 2 \frac{d}{dt}f^{2}(t)\leq h(t)f^{2}(t)+g(t)f(t).$$
		Then for ang time $t\in [0,T],$ we have
		$$f(t)\leq f(0)exp\int_0^th(\tau)d\tau+\int_0^t g(\tau)\ exp(\int_\tau ^t h(\tau)dt')d\tau.$$
	\end{lemm}
	Now we state some useful results in the transport equation theory, which are important to the proofs of our main theorem later.
	\begin{equation}\label{transport}
		\left\{\begin{array}{l}
			f_t+v\cdot\nabla f=g,\ x\in\mathbb{R}^d,\ t>0, \\
			f(0,x)=f_0(x).
		\end{array}\right.
	\end{equation}
	\begin{lemm}\label{existence}\cite{chemin2011}
		Let $1\leq p\leq p_1\leq\infty,\ 1\leq r\leq\infty,\ s> -d\min(\frac 1 {p_1}, \frac 1 {p'})$. Let $f_0\in B^s_{p,r}$, $g\in L^1([0,T];B^s_{p,r})$, and let $v$ be a time-dependent vector field such that $v\in L^\rho([0,T];B^{-M}_{\infty,\infty})$ for some $\rho>1$ and $M>0$, and
		$$
		\begin{array}{ll}
			\nabla v\in L^1([0,T];B^{\frac d {p_1}}_{p_1,\infty}), &\ \text{if}\ s<1+\frac d {p_1}, \\
			\nabla v\in L^1([0,T];B^{s-1}_{p,r}), &\ \text{if}\ s>1+\frac d {p_1}\ or\ (s=1+\frac d {p_1}\ and\ r=1).
		\end{array}
		$$
		Then the equation \eqref{transport} has a unique solution $f$ in \\
		-the space $C([0,T];B^s_{p,r})$, if $r<\infty$; \\
		-the space $\Big(\bigcap_{s'<s}C([0,T];B^{s'}_{p,\infty})\Big)\bigcap C_w([0,T];B^s_{p,\infty})$, if $r=\infty$.
	\end{lemm}
	\begin{lemm}\label{priori estimate}\cite{chemin2011,lijin2016}
		Let $s\in\mathbb{R},\ 1\leq p,r\leq\infty$.
		There exists a constant $C$ such that for all solutions $f\in L^{\infty}([0,T];B^s_{p,r})$ of \eqref{transport} in one dimension with initial data $f_0\in B^s_{p,r}$, and $g\in L^1([0,T];B^s_{p,r})$, we have for a.e. $t\in[0,T]$,
		$$ \|f(t)\|_{B^s_{p,r}}\leq \|f_0\|_{B^s_{p,r}}+\int_0^t \|g(t')\|_{B^s_{p,r}}dt'+\int_0^t V^{'} (t^{'})\|f(t)\|_{B^s_{p,r}}dt{'} $$
		or
		$$ \|f(t)\|_{B^s_{p,r}}\leq e^{CV(t)}\Big(\|f_0\|_{B^s_{p,r}}+\int_0^t e^{-CV(t')}\|g(t')\|_{B^s_{p,r}}dt'\Big) $$
		with
		\begin{equation*}
			V'(t)=\left\{\begin{array}{ll}
				\|\nabla v\|_{B^{s+1}_{p,r}},\ &\text{if}\ s>\max(-\frac 1 2,\frac 1 {p}-1), \\
				\|\nabla v\|_{B^{s}_{p,r}},\ &\text{if}\ s>\frac 1 {p}\ \text{or}\ (s=\frac 1 {p},\ p<\infty, \ r=1),
			\end{array}\right.
		\end{equation*}
		and when $s=\frac 1 p-1,\ 1\leq p\leq 2,\ r=\infty,\ \text{and}\ V'(t)=\|\nabla v\|_{B^{\frac 1 p}_{p,1}}$.\\
		If $f=v,$ for all $s>0, V^{'}(t)=\|\nabla v(t)\|_{L^{\infty}}.$
	\end{lemm}

	\begin{lemm}\label{cont1}\cite{chemin2011,lijin2016}
		Let $y_0\in B^{\frac{1}{p}}_{p,1}$ with $1\leq p<\infty,$  and $f\in L^1([0,T];B^{\frac{1}{p}}_{p,1}).$ Define $\bar{\mathbb{N}}=\mathbb{N} \cup {\infty},$ for $n\in \bar{\mathbb{N}},$ denote by $y_n \in C([0,T];B^{\frac{1}{p}}_{p,1})$ the solution of
		\begin{equation}
			\left\{\begin{aligned}
				&\partial_ty_n+A_n(u)\partial_xy_n=f,\quad x\in \mathbb{R},\\
				&y_n(t,x)|_{t=0}=y_0(x). \\
			\end{aligned} \right. \label{e1}
		\end{equation}
		Assume for some $\alpha(t)\in L^1(0,T),\  \sup\limits_{n\in \bar{\mathbb{N}}} \|A_n(u)\|_{B^{1+\frac{1}{p}}_{p,1}}\leq \alpha (t).$ If $A_n(u)$ converges in $A_{\infty}(u)$ in $L^1([0,T];B^{\frac{1}{p}}_{p,1}),$ then the sequence $(y_n)_{n\in \mathbb{N}}$ converges in $C([0,T];B^{\frac{1}{p}}_{p,1}).$
	\end{lemm} 
	\begin{defi}
		Let $T>0,\ s\in\mathbb{R},$ and $1\leq p,r \leq\infty.$ Set
		\begin{equation*}
			E^s_{p,r}(T)\triangleq \left\{\begin{array}{ll}
				C([0,T];B^s_{p,r})\cap C^1([0,T];B^{s-1}_{p,r}), & \text{if}\ r<\infty,  \\
				C_w([0,T];B^s_{p,\infty})\cap C^{0,1}([0,T];B^{s-1}_{p,\infty}), & \text{if}\ r=\infty.
			\end{array}\right.
		\end{equation*}
	\end{defi}
	
	\section{Local well-posedness}
	\par
	In this section, we use Theorem 1.1 of\cite{yeweikui2023} to prove the local well-posedness of the equation \eqref{eq1} in  Besov spaces $B^{1+\frac 1 p}_{p,1}$ with $ 1\leq p <+\infty.$ 
	Some new variables are given as follows. Note that $P(x)=\frac{1}{2}e^{-|x|}$ and $P(x)\ast f=(1-\partial_x^2)^{-1}f$ for all $f\in L^2(\mathbb{R}),$ then the equation $\eqref{eq0}$ can be rewritten as
	\begin{equation}\label{eq1}
		\left\{\begin{array}{l}
			u_t-uu_x=\partial_xP\ast (\frac{k_2}{2}\rho^{2}+\frac{k_1}{2}u^{2}+\frac{3-k_1}{2}u_{x}^{2}), \quad (t,x)\in\mathbb{R}^{+}\times\mathbb{R},\\
			\rho_t-k_3u\rho_x=k_3\rho u_x, \\
			u(0,x)=u_0(x), \rho(0,x)=\rho_0(x),\quad x\in\mathbb{R},
		\end{array}\right.
	\end{equation}
	
	\begin{theo}\label{the1}
		Let $(u_0,v_0)\in B^{1+\frac{1}{p}}_{p,1}\times B^{\frac{1}{p}}_{p,1}$ with $p\in [1,\infty).$  
		Then, there exists a time $T>0$ such that \eqref{eq1} has a solution (u,v) in $E^{1+\frac{1}{p}}_{p,1}(T)\times E^{\frac{1}{p}}_{p,1}(T)$  
		and the solution is unique when $k_3=1$. Moreover,  the solution depends continuously on the initial data.
	\end{theo}
	\begin{proof}\
		Now  we prove Theorem \ref{the1} in three steps.
		
		\textbf{First step: Existence.}
		
	Assuming $u^0=\rho^0=0,$ we define by induction a sequence $(u^n,\rho^n)_{n\in\mathbb{N}}$
		of smooth functions by solving the following two-component b-family equations:
		\begin{equation}\label{14}
			\left\{\begin{array}{l}
				u^{n+1}_t-u^{n}u^{n+1}_x=F^n,\quad  (t,x)\in\mathbb{R}^{+}\times\mathbb{R},\\
				\rho^{n+1}_t-k_3u^{n}\rho^{n+1}_x=G^n,\quad \\
				u^{n+1}|_{t=0}=S_{n+1}u_0, ~\rho^{n+1}|_{t=0}=S_{n+1}\rho_0, \quad x\in\mathbb{R},
			\end{array}\right.
		\end{equation}
		where
	$F^n=\partial_xP\ast (\frac{k_2}{2}(\rho^n)^{2}+\frac{k_1}{2}(u^n)^{2}+\frac{3-k_1}{2}(u^n_{x})^{2})$, $G^n=k_3\rho^n u^n_x.$ Assume that $(u^n,\rho^n)\in L^\infty([0,T];B^{s}_{p,r}\times B^{s-1}_{p,r})$ for all $T>0,$ it follows that
			\begin{align}
			\|F^n\|_{B^{1+\frac{1}{p}}_{p,1}}&\leq  {\frac{k_2}{2}}\|\rho^n\|^{2}_{B^{\frac{1}{p}}_{p,1}}+\frac{k_1}{2}\|u_n\|^{2}_{B^{1+\frac{1}{p}}_{p,1}}+\frac{3-k_1}{2}\|u_n\|^{2}_{B^{1+\frac{1}{p}}_{p,1}}\nonumber\\&\leq C(\|u^n\|_{B^{1+\frac{1}{p}}_{p,1}}+\|\rho^n\|_{B^{\frac{1}{p}}_{p,1}})(\|u^n\|_{B^{1+\frac{1}{p}}_{p,1}}+1)(\|\rho^n\|_{B^{\frac{1}{p}}_{p,1}}+1),
			\label{16}
		\end{align}
		\begin{align}
			\|G^n\|_{B^{\frac{1}{p}}_{p,1}}&\leq  k_3\|u^n_x\|_{B^{\frac{1}{p}}_{p,1}}\|\rho^n\|_{B^{\frac{1}{p}}_{p,1}}\leq C(\|u^n\|_{B^{1+\frac{1}{p}}_{p,1}}+\|\rho^n\|_{B^{\frac{1}{p}}_{p,1}})(\|u^n\|_{B^{1+\frac{1}{p}}_{p,1}}+1).
			\label{17}
		\end{align}
		By Lemma \ref{existence}  with \eqref{14}, 
		there exists a global solution $(u^{n+1},\rho^{n+1})\in E^{1+\frac{1}{p}}_{p,1}(T)\times E^{\frac{1}{p}}_{p,1}(T)$ for any $T>0$.  Futhermore, from Lemma \ref{priori estimate}, \eqref{16} and\eqref{17}, we have
		\begin{align}
			\|u^{n+1}(t)\|_{B^{1+\frac{1}{p}}_{p,1}}+\|\rho^{n+1}(t)\|_{B^{\frac{1}{p}}_{p,1}}
			\leq&~e^{C\int_0^t \|u^n_x\|_{B^{\frac{1}{p}}_{p,1}}dt'}
			\Big((\|u_0\|_{B^{1+\frac{1}{p}}_{p,1}}+\|\rho_0\|_{B^{\frac{1}{p}}_{p,1}})+\int_0^t e^{-C\int_0^{t'} \|u^n_x\|_{B^{\frac{1}{p}}_{p,1}} dt''}\notag\\
			&
			\times(\|u^n\|_{B^{1+\frac{1}{p}}_{p,1}}+\|\rho^n\|_{B^{\frac{1}{p}}_{p,1}})(\|u^n\|_{B^{1+\frac{1}{p}}_{p,1}}+1)(\|\rho^n\|_{B^{\frac{1}{p}}_{p,1}}+1)dt'\Big).
			\label{18}
		\end{align}
		
		For fixed  $T>0$ such that 
		$$T<\{\frac{1}{8C(\|u_0\|_{B^{1+\frac{1}{p}}_{p,1}}+\|\rho_0\|_{B^{\frac{1}{p}}_{p,1}})^2},\frac{\ln2}{C}\}$$
		 and suppose that
		\begin{equation}\label{19}
			\forall t\in [0,T],\ \|u^n\|_{B^{1+\frac{1}{p}}_{p,1}}+\|\rho^n\|_{B^{\frac{1}{p}}_{p,1}} \leq
			\frac{2(\|u_0\|_{B^{1+\frac{1}{p}}_{p,1}}+\|\rho_0\|_{B^{\frac{1}{p}}_{p,1}})}{(1-8C(\|u_0\|_{B^{1+\frac{1}{p}}_{p,1}}+\|\rho_0\|_{B^{\frac{1}{p}}_{p,1}})^2t)^\frac{1}{2}}.
		\end{equation}
		Plugging \eqref{19} into \eqref{18} and using a simple calculation yield
		\begin{align*}
			\|u^{n+1}(t)\|_{B^{1+\frac{1}{p}}_{p,1}}+\|\rho^{n+1}\|_{B^{\frac{1}{p}}_{p,1}} &\leq
			C(\|u_0\|_{B^{1+\frac{1}{p}}_{p,1}}+\|\rho_0\|_{B^{\frac{1}{p}}_{p,1}})e^{\frac{C}{2}\int_0^t(\|u^{n}(\tau)\|_{B^{1+\frac{1}{p}}_{p,1}}+\|\rho^{n}(\tau)\|_{B^{\frac{1}{p}}_{p,1}}+1)^2d\tau} \notag\\
			&\leq
			C(\|u_0\|_{B^{1+\frac{1}{p}}_{p,1}}+\|\rho_0\|_{B^{\frac{1}{p}}_{p,1}})e^{C\int_0^t(\|u^{n}(\tau)\|_{B^{1+\frac{1}{p}}_{p,1}}+\|\rho^{n}(\tau)\|_{B^{\frac{1}{p}}_{p,1}})^2+1d\tau} \notag \\
			&\leq 	\frac{2(\|u_0\|_{B^{1+\frac{1}{p}}_{p,1}}+\|\rho_0\|_{B^{\frac{1}{p}}_{p,1}})}{(1-8C(\|u_0\|_{B^{1+\frac{1}{p}}_{p,1}}+\|\rho_0\|_{B^{\frac{1}{p}}_{p,1}})^2t)^\frac{1}{2}}.
		\end{align*}
		Then $(u^n,\rho^n)_{n\in\mathbb{N}}$ is uniformly  bounded in $L^{\infty}([0,T];{B^{1+\frac{1}{p}}_{p,1}}\times B^{\frac{1}{p}}_{p,1})$.
		
		By using the compactness method for the approximating sequence $(u^n,\rho^n)_{n\in\mathbb{N}}$ to get a solution $(u,\rho)$ of \eqref{eq1}.  Due to the uniformly boundedness of $(u^n,\rho^n)$ in $L^{\infty}([0,T];B^{1+\frac{1}{p}}_{p,1}\times B^{\frac{1}{p}}_{p,1}),$ it is deduced that $(u^n_t,\rho^n_t)$ is uniformly bounded in $L^{\infty}([0,T];B^{\frac{1}{p}}_{p,1}\times B^{\frac{1}{p}-1}_{p,1} ).$ Thus,
		$$u^n\ is \ uniformly\  bounded\ in \  C([0,T];B^{1+\frac{1}{p}}_{p,1})\cap C^{\frac 1 2}([0,T];B^{\frac{1}{p}}_{p,1}),$$
			$$\rho^n\ is \ uniformly\  bounded\ in \  C([0,T];B^{\frac{1}{p}}_{p,1})\cap C^{\frac 1 2}([0,T];B^{\frac{1}{p}-1}_{p,1}).$$
		Let $(\phi_j)_{j\in\mathbb{N}}$ be a sequence of smooth functions with value in $[0,1]$ supported in $B(0,j+1)$ and value equal  to $1$ on $B(0,j).$ According to Theorem 2.94 in  \cite{chemin2011}, the map $(u^n,\rho^n)\mapsto (\phi_j u^n,\phi_j \rho^n)$ is compact from $B^{1+\frac{1}{p}}_{p,1}\times B^{\frac{1}{p}}_{p,1}$ to $B^{\frac{1}{p}}_{p,1}\times B^{\frac{1}{p}-1}_{p,1}$. Thus, up to extraction, it is known from Ascoli's theorem that the sequence $(\phi_j u^n,\phi_j \rho^n)_{j\in\mathbb{N}}$ converges to some function $(u_j,\rho_j)$ is hold. On the  other hand, 
		by Cantor's diagonal process, there exists a subsequence  of $(u_j,\rho_j)_{j\in\mathbb{N}}$ ( still mark it as $(u_j,\rho_j)_{j\in\mathbb{N}}$) such that for any $j\in \mathbb{N},$ $(\phi_j u^n,\phi_j\rho^n)$ tends to $(u_j,\rho_j)$ in $C([0,T]; B^{\frac 1 p}_{p,1}\times B^{\frac{1}{p} -1}_{p,1}).$  Since $\phi_j\phi_{j+1}=\phi_j,$   we can obtain that  $u_j=\phi_ju_{j+1}$ and $\rho_j=\phi_j\rho_{j+1}$ Therefore, we can inferred that there exists some function $u$  such that for all $\phi\in \mathcal{D}(\mathbb{R}),$ the sequence  $(\phi u^n,\phi\rho^n)_{n\geq 1}$ tends (up to a subsequence independent of $\phi$) to $                                                                                                                                                                                                                                                                                                                                                                                                                                                                                                                                                                                                                                                                                                                                                                                                                                                                                                                                                                                                                                                                                                                                                                                                                                                                                                                                                                                                                                                                                                                                                                                                                                                                                                                                                                                                                                                                                                                                                                                                                                                                                                                                                                                                                                                                                                                                                                                                                                                                                                                                                                                                                                                                                                                                                                                                                                                                                                                                                                                                                                                                                                                                                                                                                                                                                                                                                                                                                                                                                                                                                                                                                                                                                                                                                                                                                                                                                                                                                                                                                                                                                                                                                                                                                                                                                                                                                                                                                                                                                                                                                                                                                                                                                                                                                                                                                                                                                                                                                                                                                                                                                                                                                                                                                                                                                                                                                                                                                                                                                                                                                                                                                                                                                                                                                                                                                                                                                                                                                                                                                                                                                                                                                                                                                                                                                                                                                                                                                                                                                                                                                                                                                                                                                                                                                                                                                                                                                                                                                                                                                                                                                                                                                                                                                                                                                                                                                                                                                                                                                                                                                                                                                                                                                                                                                                                                                                                                                                                                                                                                                                                                                                                                                                                                                                                                                                                                                                                                                                                                                                                                                                                                                                                                                                                                                                                                                                                                                                                                                                                                                                                                                                                                                                                                                                                                                                                                                                                                                                                                                                                                                                                                                                                                                                                                                                                                                                                                                                                                                                                                                                                                                                                                                                                                                                                                                                                                                                                                                                                                                                                                                                                                                                                                                                                                                                                                                                                                                                                                                                                                                                                                                                                                                                                                                                                                                                                                                                                                                                                                                                                                                                                                                                                                                                                                                                                                                                                                                                                                                                                                                                                                                                                                                                                                                                                                                                                                                                                                                                                                                                                                                                                                                                                                                                                                                                                                                                                                                                                                                                                                                                                                                                                                                                                                                                                                                                                                                                                                                                                                                                                                                                                                                                                                                                                                                                                                                                                                                                                                                                                                                                                                                                                                                                                                                                                                                                                                                                                                                                                                                                                                                                                                                                                                                                                                                                                                                                                                                                                                                                                                                                                                                                                                                                                                                                                                                                                                                                                                                                                                                                                                                                                                                                                                                                                                                                                                                                                                                                                                                                                                                                                                                                                                                                                                                                                                                                                                                                                                                                                                                                                                                                                                                                                                                                                                                                                                                                                                                                                                                                                                                                                                                                                                                                                                                                                                                                                                                                                                                                                                                                                                                                                                                                                                                                                                                                                                                                                                                                                                                                                                                                                                                                                                                                                                                                                                                                                                                                                                                                                                                                                                                                                                                                                                                                                                                                                                                                                                                                                                                                                                                                                                                                                                                                                                                                                                                                                                                                                                                                                                                                                                                                                                                                                                                                                                                                                                                                                                                                                                                                                                                                                                                                                                                                                                                                                                                                                                                                                                                                                                                                                                                                                                                                                                                                                                                                                                                                                                                                                                                                                                                                                                                                                                                                                                                                                                                                                                                                                                                                                                                                                                                                                                                                                                                                                                                                                                                                                                                                                                                                                                                                                                                                                                                                                                                                                                                                                                                                                                                                                                                                                                                                                                                                                                                                                                                                                                                                                                                                                                                                                                                                                                                                                                                                                                                                                                                                                                                                                                                                                                                                                                                                                                                                                                                                                                                                                                                                                                                                                                                                                                                                                                                                                                                                                                                                                                                                                                                                                                           (u,\rho)$ in $C([0,T];B^{\frac{1}{p}}_{p,1})\times B^{\frac{1}{p}-1}_{p,1}).$ 
		Since the uniform boundness of $u_n$ and Fatou property for Besov spaces, we obtain that $(u,\rho)\in L^{\infty}([0,T];B^{1+\frac{1}{p}}_{p,1}\times B^{\frac{1}{p}}_{p,1}).$ By  virtue of interpolation, we get $(\phi u^n,\phi \rho^n)$ tends to $(u,\rho)$ in $C([0,T];B^{1+\frac{1}{p}-\varepsilon}_{p,1}\times B^{\frac{1}{p}-\varepsilon}_{p,1})$ for any $\varepsilon>0.$  Below we show that $(u,\rho)$ is the solution of $\eqref{eq1}.$
		 	Indeed, for any $\psi \in B^{1-\frac1 p}_{p',\infty},$ we only need to prove that
		\begin{align}\label{oooo}
			\left\langle(\phi u^n)_t-(\phi u)_t, \psi\right\rangle-\left\langle(\phi u^n)(\phi u^n)_x-(\phi u)(\phi u)_x, \psi\right\rangle\nonumber\\
			-\left\langle F(\phi u^n,\phi\rho^n)-F(\phi u,\phi\rho), \psi\right\rangle\rightarrow 0,\quad as~n\rightarrow \infty.
		\end{align}
			\begin{align}\label{vv}
			\left\langle(\phi \rho^n)_t-(\phi\rho)_t, \psi\right\rangle-\left\langle(\phi u^n)(\phi \rho^n)_x-(\phi u)(\phi\rho)_t, \psi\right\rangle\nonumber\\
			-\left\langle G(\phi u^n,\phi\rho^n)-G(\phi u,\phi\rho), \psi\right\rangle\rightarrow 0,\quad as~n\rightarrow \infty.
		\end{align}
		The method of handling each term is similar. For the sake of conciseness, we treat only the term of $\left\langle(\phi u^n)(\phi u^n)_x-(\phi u)(\phi u)_x, \psi\right\rangle $.
		\begin{align*}
		&|\left\langle(\phi u^n)(\phi u^n)_x-(\phi u)(\phi u)_x, \psi\right\rangle|\nonumber\\
		=&|\left\langle(\phi u^n-\phi u)(\phi u^n)_x+((\phi u^n)_x-(\phi u)_x)(\phi u), \psi\right\rangle|\nonumber\\
			\leq& C\|(\phi u^n-\phi u)(\phi u^n)_x+((\phi u^n)_x-(\phi u)_x)(\phi u)\|_{B_{p,1}^{\frac{1}{p}-1}}\|\psi\|_{B_{p',\infty}^{1-\frac{1}{p}}}\nonumber\\
			\leq& C\|(\phi u^n-\phi u)(\phi u^n)_x+((\phi u^n)_x-(\phi u)_x)(\phi u)\|_{B_{p,1}^{\frac{1}{p}-\varepsilon}}\|\psi\|_{B_{p',\infty}^{1-\frac{1}{p}}}\nonumber\\
		\leq& C\|\phi u^n-\phi u\|_{B_{p,1}^{\frac{1}{p}}}\|(\phi u^n)_x\|_{B_{p,1}^{\frac{1}{p}}}+\|(\phi u^n)_x-(\phi u)_x\|_{B_{p,1}^{\frac{1}{p}-\varepsilon}}\|\phi u\|_{B_{p,1}^{\frac{1}{p}-\varepsilon}}\|\psi\|_{B_{p',\infty}^{1-\frac{1}{p}}}\nonumber\\	
		\leq& C\|\phi u^n-\phi u\|_{B_{p,1}^{\frac{1}{p}}}\|\phi u^n\|_{B_{p,1}^{1+\frac{1}{p}}}+\|\phi u^n-\phi u\|_{B_{p,1}^{1+\frac{1}{p}-\varepsilon}}\|\phi u\|_{B_{p,1}^{1+\frac{1}{p}}}\|\psi\|_{B_{p',\infty}^{1-\frac{1}{p}}}.		
		\end{align*}
		Using the fact that  $\phi u^n \rightarrow \phi u$ in $C([0,T],B^{\frac1 p}_{p,1})$, $C([0,T],B^{1+\frac1 p-\varepsilon}_{p,1})$  and  $\phi u^n \in B^{1+\frac{1}{p}}_{p,1},\ \psi\in B^{1-\frac1 p}_{p',\infty},$  we conclude that $|\left\langle(\phi u^n)(\phi u^n)_x-(\phi u)(\phi u)_x, \psi\right\rangle|\rightarrow 0$ when $n\rightarrow \infty.$
	Similarly, we can prove that	\eqref{oooo} and \eqref{vv} hold.
		Then, we get $(u,\rho)$ satisfies equation \eqref{eq1}. Finally, using the equation again, we see that $((\phi u)_t,(\phi\rho)_t)\in C([0,T];B_{p,1}^{\frac{1}{p}}\times B_{p,1}^{\frac{1}{p}-1})$. Thus, $(u,\rho)$ belong to $E^{1+\frac{1}{p}}_{p,1}(T)\times E^{\frac{1}{p}}_{p,1}(T).$
	
		\textbf{Second step. Uniqueness}		
		
		In this step, we prove the uniqueness of the solution when $k_3=1$ in the Lagrangian coordinate. 
		
		Let $u$ is a smooth solution, assume that $y: t\mapsto y(t,\cdot)$ be the characteristic as 
		\begin{equation}
			\left\{\begin{aligned}
				&\frac{d}{dt}y(t,\xi)=u(t,-y(t,\xi)),\quad x\in \mathbb{R},\\
				&y(t,\xi)|_{t=0}=\bar{y}(\xi). \\
			\end{aligned} \right. \label{02}
		\end{equation}
		Define $U(t,\xi)=u(t,-y(t,\xi)),$ $V(t,\xi)=\rho(t,-y(t,\xi)),$ then $U_{\xi}(t,\xi)=-u_x(t,-y(t,\xi))y_{\xi}(t,\xi).$ From the definition of the characteristic \eqref{02} and \eqref{eq1}, we get that $U(t,\xi)$, $V(t,\xi)$ satisfies the following equation
		\begin{align}\label{0.5}
			U_t(t,\xi)=u_t(t,y(t,\xi))-y_t(t,\xi)u_x(t,y(t,\xi))=\tilde{F}(t,\xi),
		\end{align} 
		\begin{align}\label{05}
			V_t(t,\xi)=\rho_t(t,y(t,\xi))-y_t(t,\xi)\rho_x(t,y(t,\xi))=-V\frac{U_\xi}{y_\xi}(t,\xi),
		\end{align} 
		where
		\begin{align}
			&\tilde{F}(t,\xi)=\frac{1}{2}	\int_{-\infty}^{\infty}{\rm sign}(-y(t,\xi)-x) e^{-|-y(t,\xi)-x|}(\frac{k_2}{2}\rho^2+\frac{k_1}{2}u^2+u^2_x)(t,x)dx.\label{04}
		\end{align}
		
		We perform the change of variables $x=-y(t,\eta),$ it is can obtain an expression for $\tilde{F}(t,\xi)$ in terms  of the new variable $\xi$ and  by \eqref{04}, it follows that
		\begin{align}
			&\tilde{F}(t,\xi)=-\frac{1}{2}	\int_{-\infty}^{\infty}{\rm sign}(-y(t,\xi)+y(t,\eta)) e^{-|-y(t,\xi)+y(t,\eta)|}(\frac{k_2}{2}V^2y_\eta+\frac{k_1}{2}U^2y_\eta+\frac{U_\eta^2}{y_\eta})(t,\eta)d\eta.\label{09}
		\end{align}
		Taking the derivative of $\tilde{F}(t,\xi)$ with respect to variable $\xi,$ we can get the following
		\begin{align}\label{06}
			\tilde{F}_{\xi}(t,\xi)=\frac{k_2}{2}V^2y_\xi+\frac{k_1}{2}U^2y_\xi+\frac{U_\xi^{2}}{y_\xi}(t,\xi)-\frac{1}{2}\int_{-\infty}^{\infty} e^{-|-y(t,\xi)+y(t,\eta)|}(\frac{k_2}{2}V^2y_\eta+\frac{k_1}{2}U^2y_\eta+\frac{U_\eta^2}{y_\eta})(t,\eta)d\eta
		\end{align}
		Note that
		\begin{align}\label{07}
			\frac{d}{dt}U_{\xi}=\tilde{F}_{\xi}(t,\xi).
		\end{align}
		Similarly, we have 
		\begin{align}\label{1.7}
			\frac{d}{dt} y_{\xi}=U_{\xi}(t,\xi).
		\end{align}   
		Notice that $y(t,\xi)$ and $y_{\xi}(t,\xi)$ satisfy the following integral form    
		\begin{align}
			&y(t,\xi)=\xi-\int_0^tU(\tau,\xi)d\tau,\label{b0}\\
			&y_{\xi}=1-\int_0^t U_{\xi}(\tau,\xi)d\tau.\label{b1}
		\end{align}
		
		From  \textbf{Step 1}, we get that $u$ is uniformly bounded in $C_T(B^{1+\frac{1}{p}}_{p,1}).$ Since  $B^{1+\frac{1}{p}}_{p,1}\hookrightarrow W^{1,p}\cap W^{1,\infty},$ we get $u\in  C_T(W^{1,p}\cap W^{1,\infty} ).$ Furthermore, we see that $y_{\xi}$ is uniformly bounded in $L^{\infty}_T(L^{\infty}).$  Indeed, by \eqref{b1} and Gronwall's inequality to get the following
		\begin{align}\label{20}
			\|y_{\xi}\|_{L^{\infty}}\leq e^{Ct\|U\|_{W^{1,\infty}}}\|\bar{y}_{\xi}(\xi)\|_{L^{\infty}}, t\in[0,T].	
		\end{align}
		On the other hand, we get $\frac{1}{2}\leq y_{\xi}\leq C_{u_0}$ for $T>0$ small enough. Without loss of generality,  assume that $t$ is sufficiently small, if not we use the continuous method. We can also claim that $U(t,\xi)\in L^{\infty}(W^{1,p}).$
		\begin{align*}
			&\|U(t,\xi)\|^p_{L^p}=\int_{\mathbb{R}} |U(t,\xi)|^pd\xi=\int_{\mathbb{R}} |u(t,-y(t,\xi))|^p\frac{1}{y_{\xi}}dy\leq \|u\|^p_{L^p}\|\frac{1}{y_{\xi}}\|_{L^{\infty}}\leq 2\|u\|^p_{L^p}\leq C;\\
			&\|U_{\xi}(t,\xi)\|^p_{L^p}=\int_{\mathbb{R}} |U_{\xi}(t,\xi)|^p|y_{\xi}|^{p-1}d\xi=\int_{\mathbb{R}} |u_x(t,-y(t,\xi))|^p|y_{\xi}|^{p-1}dy\leq \|u_x\|^p_{L^p}\|{y_{\xi}}\|^{p-1}_{L^{\infty}}\leq C^{p-1}_{u_0}\|u_x\|^p_{L^p}\leq C.
		\end{align*}
		Thus, we get $U(t,\xi)\in L^{\infty}_T(W^{1,p}\cap W^{1,\infty}), y(t,\xi)-\xi\in L^{\infty}_T(W^{1,p}\cap W^{1,\infty})$ and $\frac{1}{2}\leq y_{\xi}\leq C_{u_0}$ for satisfes \eqref{0.5} and \eqref{05} for any $t\in [0,T].$
		
		We now prove the uniqueness. Let $(u_1,\rho_1),(u_2,\rho_2)$ are two solutions to \eqref{eq1}, thus $U_i(t,\xi)=u_i(t,-y_i(t,\xi))$, $V_i(t,\xi)=\rho_i(t,-y_i(t,\xi))\ (i=1,2)$ satifies 
		\begin{align}
			&\frac{d}{dt}U_i(t,\xi)=u_{it}(t,-y_i)+y_{it}(t,\xi)u_{ix}(t,-y_i(t,\xi))=\tilde{F}_i(t,\xi).\label{21}
		\end{align}
		Similarly, we have $U_i(t,y_i(t,\xi))\in L^{\infty}_T(W^{1,p}\cap W^{1,\infty}), y_i(t,\xi)-\xi\in L^{\infty}_T(W^{1,p}\cap W^{1,\infty})$ and $\frac{1}{2}\leq y_{i\xi}\leq C_{u_0}$ for sufficiently small $T.$
		
		We now estimate $\|\tilde{F}_{1}(t,\xi)-\tilde{F}_{2}(t,\xi)\|_{W^{1,p}\cap W^{1,\infty}}.$
		\begin{align}\label{q1}
			\tilde{F}_{1}(t,\xi)-\tilde{F}_{2}(t,\xi)=&-\frac{1}{2}	\int_{-\infty}^{\infty}{\rm sign}(-y_1(t,\xi)+y_1(t,\eta)) e^{-|-y_1(t,\xi)+y_1(t,\eta)|}(\frac{k_2}{2}V_1^2{y_1}_\eta+\frac{k_1}{2}U_1^2{y_1}_\eta+\frac{{U_1}_\eta^2}{{y_1}_\eta})(t,\eta)d\eta \notag\\
			+&\frac{1}{2}	\int_{-\infty}^{\infty}{\rm sign}(-y_2(t,\xi)+y_2(t,\eta)) e^{-|-y_2(t,\xi)+y_2(t,\eta)|}(\frac{k_2}{2}V_2^2{y_1}_\eta+\frac{k_1}{2}U_2^2{y_2}_\eta+\frac{{U_2}_\eta^2}{{y_2}_\eta})(t,\eta)d\eta.
		\end{align}
	First, We estimate $\|\tilde{F}_{1}(t,\xi)-\tilde{F}_{2}(t,\xi)\|_{L^{p}\cap L^{\infty}}.$ The main difficult is to estimate
	\begin{align}
		\Omega(t,\xi)=&-\frac{1}{2}	\int_{-\infty}^{\infty}{\rm sign}(-y_1(t,\xi)+y_1(t,\eta)) e^{-|-y_1(t,\xi)+y_1(t,\eta)|}\frac{{U_1}_\eta^2}{{y_1}_\eta}(t,\eta)d\eta \notag\\
&+\frac{1}{2}	\int_{-\infty}^{\infty}{\rm sign}(-y_2(t,\xi)+y_2(t,\eta)) e^{-|-y_2(t,\xi)+y_2(t,\eta)|}\frac{{U_2}_\eta^2}{{y_2}_\eta}(t,\eta)d\eta.
	\end{align}
	Note that $y_{i}(t,\xi)(i=1,2)$ is
		monotonically increasing, thus ${\rm sign}(y_i(t,\xi)-y_i(t,\eta))={\rm sign}(\xi-\eta).$ Futhermore, we have
		\begin{align}\label{q2}
				\Omega(t,\xi)=&-\frac{1}{2}	\int_{-\infty}^{\infty}{\rm sign}(-\xi+\eta) (e^{-|-y_1(t,\xi)+y_1(t,\eta)|}-e^{-|-y_2(t,\xi)+y_2(t,\eta)|})\frac{{U_1}_\eta^2}{{y_1}_\eta}(t,\eta)d\eta \notag\\
				&+\frac{1}{2}	\int_{-\infty}^{\infty}{\rm sign}(-\xi+\eta) e^{-|-y_2(t,\xi)+y_2(t,\eta)|}(\frac{{U_2}_\eta^2}{{y_2}_\eta}(t,\eta)-\frac{{U_1}_\eta^2}{{y_1}_\eta}(t,\eta))d\eta\notag\\
			=&I+II.
		\end{align}
		If $\xi>\eta(or \ \xi<\eta),$ then $y_i(t,\xi)>y_i(t,\eta)(or  \ y_i(t,\xi)<y_i(t,\eta)).$ Thus, we have
		\begin{align}\label{23}
			I=&\int_{-\infty}^{\xi}	(e^{y_1(t,\eta)-y_1(t,\xi)}-e^{y_2(t,\eta)-y_2(t,\xi)})\frac{{U_1}_\eta^2}{{y_1}_\eta}(t,\eta)d\eta-\int_{\xi}^{\infty}(e^{y_1(t,\xi)-y_1(t,\eta)}-e^{y_2(t,\xi)-y_2(t,\eta)})\frac{{U_1}_\eta^2}{{y_1}_\eta}(t,\eta)d\eta \notag\\
			=&\int_{-\infty}^{\xi}	e^{\eta-\xi}(e^{\int_0^t(U_1(\tau,\eta)-U_1(\tau,\xi))d\tau}-e^{\int_0^t(U_2(\tau,\eta)-U_2(\tau,\xi))d\tau})\frac{{U_1}_\eta^2}{{y_1}_\eta}(t,\eta)d\eta\notag\\
			&-\int^{\infty}_{\xi}	e^{\xi-\eta}(e^{\int_0^t(U_1(\tau,\xi)-U_1(\tau,\eta))d\tau}-e^{\int_0^t(U_2(\tau,\xi)-U_2(\tau,\eta))d\tau})\frac{{U_1}_\eta^2}{{y_1}_\eta}(t,\eta)d\eta\notag\\
			\leq&C\|U_1-U_2\|_{L_\infty}\bigg[\int_{-\infty}^{\xi}	e^{\eta-\xi}\frac{{U_1}_\eta^2}{{y_1}_\eta}(t,\eta)d\eta+\int_{\xi}^{\infty}e^{\xi-\eta}\frac{{U_1}_\eta^2}{{y_1}_\eta}(t,\eta)d\eta\bigg]\notag\\
		\leq&C\|U_1-U_2\|_{L_\infty}\bigg[1_{\leq 0}(x)e^{-|x|}\ast\frac{{U_1}_\eta^2}{{y_1}_\eta}+1_{\geq 0}(x)e^{-|x|}\ast\frac{{U_1}_\eta^2}{{y_1}_\eta}\bigg].
		\end{align}
	In the same way, we see
		\begin{align}\label{2.3}
		II\leq&C\bigg[1_{\leq 0}(x)e^{-|x|}\ast\big(|U_{1\eta}-U_{2\eta}|+|y_{1\eta}-y_{2\eta}|\big)+1_{\geq 0}(x)e^{-|x|}\ast\big(|U_{1\eta}-U_{2\eta}|+|y_{1\eta}-y_{2\eta}|\big)\bigg].
	\end{align}
	Combining with the inequalities \eqref{q2}-\eqref{2.3}, we find
	In the same way, we see
	\begin{align}\label{2.4}
	\Omega(t,\xi)\leq&C\|U_1-U_2\|_{L_\infty}\bigg[1_{\leq 0}(x)e^{-|x|}\ast\frac{{U_1}_\eta^2}{{y_1}_\eta}+1_{\geq 0}(x)e^{-|x|}\ast\frac{{U_1}_\eta^2}{{y_1}_\eta}\bigg]\notag\\
	&+\bigg[1_{\leq 0}(x)e^{-|x|}\ast\big(|U_{1\eta}-U_{2\eta}|+|y_{1\eta}-y_{2\eta}|\big)+1_{\geq 0}(x)e^{-|x|}\ast\big(|U_{1\eta}-U_{2\eta}|+|y_{1\eta}-y_{2\eta}|\big)\bigg].
	\end{align}
		Similarly, we have
		\begin{align}\label{24}
			\|\tilde{F}_{1}(t,\xi)-\tilde{F}_{2}(t,\xi)\|_{L^{\infty}\cap L^p}\leq& C(\|U_1-U_2\|_{L^{\infty}\cap L^p}+\|U_{1\eta}-U_{2\eta}\|_{L^{\infty}\cap L^p}\notag\\
			&+\|V_{1}(t,\xi)-V_2(t,\xi)\|_{L^{\infty}\cap L^p}+\|y_{1\eta}-y_{2\eta}\|_{L^{\infty}\cap L^p}),
		\end{align}
		and
		\begin{align}\label{25}
			\|\tilde{F}_{1\xi}(t,\xi)-\tilde{F}_{2\xi}(t,\xi)\|_{L^{\infty}\cap L^p}\leq& C(\|U_1-U_2\|_{L^{\infty}\cap L^p}+\|U_{1\eta}-U_{2\eta}\|_{L^{\infty}\cap L^p}\notag\\
			&+\|V_{1}(t,\xi)-V_2(t,\xi)\|_{L^{\infty}\cap L^p}+\|y_{1\eta}-y_{2\eta}\|_{L^{\infty}\cap L^p}).
		\end{align}
		Combining \eqref{24} and \eqref{25}, we can deduce that
		\begin{align}\label{26}
			\|\tilde{F}_{1}(t,\xi)-\tilde{F}_{2}(t,\xi)\|_{W^{1,p}\cap W^{1,\infty}}\leq C(\|U_1-U_2\|_{W^{1,p}\cap W^{1,\infty}}+\|V_{1}(t,\xi)-V_2(t,\xi)\|_{L^{\infty}\cap L^p}+\|y_1-y_2\|_{W^{1,p}\cap W^{1,\infty}}).
		\end{align}
		Futhermore,
		\begin{align}\label{27}
			&\|U_1-U_2\|_{W^{1,p}\cap W^{1,\infty}}+\|V_{1}(t,\xi)-V_2(t,\xi)\|_{L^{\infty}\cap L^p}+\|y_1-y_2\|_{W^{1,p}\cap W^{1,\infty}}\notag\\&\leq C (\|U_1(0)-U_2(0)\|_{W^{1,p}\cap W^{1,\infty}}+\|V_{1}(0)(t,\xi)-V_2(0)(t,\xi)\|_{L^{\infty}\cap L^p}+\|y_1(0)-y_2(0)\|_{W^{1,p}\cap W^{1,\infty}})\notag\\
			&+C\int_0^T(\|U_1-U_2\|_{W^{1,p}\cap W^{1,\infty}}+\|V_{1}(t,\xi)-V_2(t,\xi)\|_{L^{\infty}\cap L^p}+\|y_1-y_2\|_{W^{1,p}\cap W^{1,\infty}})dt.
		\end{align}
		By the Gronwall inequality, 
		\begin{align}
			&\|U_1-U_2\|_{W^{1,p}\cap W^{1,\infty}}+\|V_{1}(t,\xi)-V_2(t,\xi)\|_{L^{\infty}\cap L^p}+\|y_1-y_2\|_{W^{1,p}\cap W^{1,\infty}}\notag\\&\leq  e^{CT}(\|U_1(0)-U_2(0)\|_{W^{1,p}\cap W^{1,\infty}}+\|V_{1}(0)(t,\xi)-V_2(0)(t,\xi)\|_{L^{\infty}\cap L^p}+\|y_1(0)-y_2(0)\|_{W^{1,p}\cap W^{1,\infty}})\notag\\&\leq e^{CT}(\|U_1(0)-U_2(0)\|_{W^{1,p}\cap W^{1,\infty}}+\|V_{1}(0)(t,\xi)-V_2(0)(t,\xi)\|_{L^{\infty}\cap L^p}+0)\notag\\&\leq e^{CT} (\|u_1(0)-u_2(0)\|_{B^{1+\frac{1}{p}}_{p,1}}+\|\rho_1(0)-\rho_2(0)\|_{B^{\frac{1}{p}}_{p,1}}),
		\end{align}
		where we use the fact is $y_1(0)=y_2(0)=\xi.$ It follows that
		\begin{align*}
			\|u_1-u_2\|_{L^p}+\|\rho_1-\rho_2\|_{L^p}&\leq C \|u_1 \circ y_1-	u_2 \circ y_2\|_{L^p}+\|V_1-V_2\|_{L^p}\notag\\&\leq C\|u_1 \circ y_1-u_2 \circ y_1+u_2 \circ y_1-u_2 \circ y_2\|_{L^p}+\|V_1-V_2\|_{L^p} \notag\\&\leq C\|U_1-U_2\|_{L^p}+C\|u_{2x}\|_{L^{\infty}}\|y_1-y_2\|_{L^p}+\|V_1-V_2\|_{L^p}\notag\\&\leq C (\|u_1(0)-u_2(0)\|_{B^{1+\frac{1}{p}}_{p,1}}+\|\rho_1(0)-\rho_2(0)\|_{B^{\frac{1}{p}}_{p,1}}).
		\end{align*}
		By the embedding $L^p\hookrightarrow B^{0}_{p,\infty},$
		\begin{align*}
			\|u_1-u_2\|_{B^{0}_{p,\infty}}+\|\rho_1-\rho_2\|_{B^{0}_{p,\infty}}&\leq C(\|u_1-u_2\|_{L^p}+\|\rho_1-\rho_2\|_{L^p})\notag\\
			&\leq C(\|u_1(0)-u_2(0)\|_{B^{1+\frac{1}{p}}_{p,1}}+\|\rho_1(0)-\rho_2(0)\|_{B^{\frac{1}{p}}_{p,1}}).
		\end{align*}
		Therefore, if $u_1(0)=u_2(0)$ and $\rho_1(0)=\rho_2(0)$, then we get the unique conclusion of solution.\\
		
		\textbf{Third step. The continuous dependence}
		
		Let $(u^n,\rho^n)$,  $(u^\infty,\rho^\infty)$ be the solutions of \eqref{eq1} with the initial data $(u^n_0,\rho^n_0),(u^\infty_0,\rho^\infty_0)$ and $(u^n_0,\rho^n_0)$ tends to $(u^\infty_0,\rho^{\infty}_0)$ in $B^{1+\frac 1 p }_{p,1}\times B^{\frac 1 p }_{p,1}.$ Combining \textbf{Step 1} and \textbf{Step 2}, we get $(u^n,\rho^n),(u^{\infty},\rho^\infty)$ which is uniformly bounded in $L^{\infty}_T(B^{1+\frac1 p }_{p,1}\times B^{\frac1 p }_{p,1})$ and
		\begin{align*}&
			\|(u^n-u^{\infty})(t)\|_{B^{0 }_{p,\infty}}\leq C\|u^n_0-u^{\infty}_0\|_{B^{1+\frac1 p }_{p,1}},\notag\\
			&\|(\rho^n-\rho^{\infty})(t)\|_{B^{0 }_{p,\infty}}\leq C\|\rho^n_0-\rho^{\infty}_0\|_{B^{\frac1 p }_{p,1}}.
		\end{align*}
		it is inferred that $(u^n,\rho^n)$ tends to $(u^{\infty},\rho^\infty)$ in $C([0,T],B^{0 }_{p,\infty}\times B^{0 }_{p,\infty}).$  According to the interpolation inequality, we claim that $(u^n,\rho^n)\rightarrow (u^{\infty},\rho^\infty)$ in $C([0,T],B^{1+\frac{1}{p}-\varepsilon}\times B^{\frac{1}{p}-\varepsilon}_{p,1})$ for any $\varepsilon >0$. When $\varepsilon =1$,  there are the following
		\begin{align}\label{a0}
			(u^n,\rho^n)\rightarrow (u^\infty,\rho^{\infty}) \ in \ C([0,T],B^{\frac{1}{p}}_{p,1}\times B^{\frac{1}{p}-1}_{p,1}).
		\end{align}
		From the above, we only need to prove $(u^n_x,\rho^n)\rightarrow(u^{\infty}_x,\rho^\infty)$ in $C([0,T],{B^{\frac1 p }_{p,1}}\times B^{\frac1 p }_{p,1}).$ Let $u_x^n=a^n+b^n,\rho^n=\omega^n+z^n$ and $(a^n,b^n)$, $(\omega^n,z^n)$ satisfying
		\begin{equation*}
			\left\{\begin{aligned}
				&\omega_t^n-k_3u^n\omega_x^n=k_3\rho^\infty u_x^\infty, \\
				&a_t^n-u^na_x^n=\partial_{xx}P\ast\bigg[\frac{k_2}{2}(\rho^\infty)^2+\frac{k_1}{2}(u^\infty)^2+\frac{3-k_1}{2}(u_x^\infty)^2\bigg],\\
				&\omega^n(0,x)=\rho_0^\infty,a^n(0,x)=\partial_x u_0^\infty \\
			\end{aligned}\right.
		\end{equation*}
		and 
		\begin{equation*}  
			\left\{\begin{aligned}
				&z_t^n-k_3u^nz_x^n=k_3\rho^nu_x^n-k_3\rho^\infty u_x^\infty, \\
				&b_t^n-u^nb_x^n=\partial_{xx}P\ast\bigg[\frac{k_2}{2}((\rho^n)^2-(\rho^\infty)^2)+\frac{k_1}{2}((u^n)^2-(u^\infty)^2)+\frac{3-k_1}{2}((u_x^n)^2-(u_x^\infty)^2)\bigg]\\
				&z^n(0,x)=\rho_0^n-\rho_0^\infty,b^n(0,x)=\partial_xu_0^n-\partial_xu_0^\infty.
			\end{aligned}\right.
		\end{equation*} 
		
		By the uniformly bounded of $(u^n, \rho^n),(u^{\infty},\rho^{\infty})$ uniformly bounded in $B^{1+\frac{1}{p}}_{p,1}\times B^{\frac{1}{p}}_{p,1},$  we have
		\begin{align}\label{eqc.5}
			&\|\rho^\infty u_x^\infty\|_{B_{p,1}^{\frac{1}{p}}}\leq C \|\rho^\infty\|_{B_{p,1}^{\frac{1}{p}}}\|u_x^\infty\|_{B_{p,1}^{\frac{1}{p}}}\leq C\|\rho^\infty\|_{B_{p,1}^{\frac{1}{p}}}\|u^\infty\|_{B_{p,1}^{1+\frac{1}{p}}},
		\end{align}
			\begin{align}\label{eqc.6}
			&\|\partial_{xx} P\ast(\frac{k_2}{2}(\rho^\infty)^2+\frac{k_1}{2}(u^\infty)^2+\frac{3-k_1}{2}(u_x^\infty)^2)\|_{B_{p,1}^{\frac{1}{p}}}\notag\\
			=&\|-(\frac{k_2}{2}(\rho^\infty)^2+\frac{k_1}{2}(u^\infty)^2+\frac{3-k_1}{2}(u_x^\infty)^2)+P\ast(\frac{k_2}{2}(\rho^\infty)^2+\frac{k_1}{2}(u^\infty)^2+\frac{3-k_1}{2}(u_x^\infty)^2)\|_{B_{p,1}^{\frac{1}{p}}}\notag\\
			\leq& C\|\rho^\infty\|^2_{B_{p,1}^{\frac{1}{p}}}+\|u^\infty\|^2_{B_{p,1}^{1+\frac{1}{p}}}.			
		\end{align}
		By Lemma \ref{cont1}, we get 
		\begin{align}\label{a1}
			(a^n,\omega^n) \rightarrow (a^\infty,\omega^{\infty}) \  in \  C([0,T],B^{\frac{1}{p}}_{p,1}\times B^{\frac{1}{p}}_{p,1}).
		\end{align}
		More importantly, since $(u^n, \rho^n),(u^{\infty},\rho^{\infty})$ uniformly bounded in $B^{1+\frac{1}{p}}_{p,1}\times B^{\frac{1}{p}}_{p,1},$  we have
		\begin{align}\label{ca.u}
		\|\rho^n u_x^n-\rho^\infty u_x^\infty\|_{B^{\frac{1}{p}}_{p,1}}	
		&=\|\rho^n u_x^n-\rho^\infty u_x^n+\rho^\infty u_x^n-\rho^\infty u_x^\infty\|_{B^{\frac{1}{p}}_{p,1}}\notag\\
		&\leq \|u_x^n(\rho^n-\rho^\infty)_{B^{\frac{1}{p}}_{p,1}}+\|\rho^\infty(u^n_x-u_x^\infty)\|_{B^{\frac{1}{p}}_{p,1}}\notag\\
		&\leq C(\|\rho^n-\rho^\infty\|_{B^{\frac{1}{p}}_{p,1}}+\|u^n_x-u_x^\infty\|_{B^{\frac{1}{p}}_{p,1}})\notag\\
		&\leq C(\|\omega^n-\omega^\infty\|_{B^{\frac{1}{p}}_{p,1}}+\|z^n-z^\infty\|_{B^{\frac{1}{p}}_{p,1}}+\|a^n-a^\infty\|_{B^{\frac{1}{p}}_{p,1}}+\|b^n-b^\infty\|_{B^{\frac{1}{p}}_{p,1}})
		\end{align}
		and
		\begin{align}\label{ca1u}
			&\bigg\|\frac{k_2}{2}((\rho^n)^2-(\rho^\infty)^2)+\frac{k_1}{2}((u^n)^2-(u^\infty)^2)+\frac{3-k_1}{2}((u_x^n)^2-(u_x^\infty)^2)\bigg\|_{B^{\frac{1}{p}}_{p,1}}	\notag\\
			\leq& C(\|\rho^n-\rho^\infty\|_{B^{\frac{1}{p}}_{p,1}}+\|u^n-u^\infty\|_{B^{\frac{1}{p}}_{p,1}}+\|u^n_x-u_x^\infty\|_{B^{\frac{1}{p}}_{p,1}})\notag\\
			\leq& C(\|\omega^n-\omega^\infty\|_{B^{\frac{1}{p}}_{p,1}}+\|z^n-z^\infty\|_{B^{\frac{1}{p}}_{p,1}}+\|u^n-u^\infty\|_{B^{\frac{1}{p}}_{p,1}}+\|a^n-a^\infty\|_{B^{\frac{1}{p}}_{p,1}}+\|b^n-b^\infty\|_{B^{\frac{1}{p}}_{p,1}}).
		\end{align}
		By \eqref{ca.u}, \eqref{ca1u} and virtue of Gronwall's inequality, we can obtain
		\begin{align}\label{cau1}
			\|b_n(t)\|_{B^{\frac{1}{p}}_{p,1}}+	\|z_n(t)\|_{B^{\frac{1}{p}}_{p,1}} \leq& Ce^{Ct} \Big(\|\rho^0_n-\rho^0_{\infty}\|_{B^{\frac{1}{p}}_{p,1}}+\|\partial_x u^0_n-\partial_x u^0_{\infty}\|_{B^{\frac{1}{p}}_{p,1}}+\int_0^t e^{-Ct'}  (\|\omega^n-\omega^\infty\|_{B^{\frac{1}{p}}_{p,1}}\notag\\
			&+\|z^n-z^\infty\|_{B^{\frac{1}{p}}_{p,1}}+\|u^n-u^\infty\|_{B^{\frac{1}{p}}_{p,1}}+\|a^n-a^\infty\|_{B^{\frac{1}{p}}_{p,1}}+\|b^n-b^\infty\|_{B^{\frac{1}{p}}_{p,1}})dt'\Big).
		\end{align}
		Since $(a^n,\omega^n)\rightarrow (a^\infty,\omega^{\infty}) \ in \ C([0,T],B^{\frac{1}{p}}_{p,1}\times B^{\frac{1}{p}}_{p,1})$ and \eqref{a0}-\eqref{a1},  we  have $(b^n,z^n)\rightarrow (0,0)$ in  $B^{\frac{1}{p}}_{p,1}\times B^{\frac{1}{p}}_{p,1}.$ Noticing that $b^{\infty}=z^{\infty}=0$ and Lemmas \ref{existence}-\ref{priori estimate}, we see that $(b^n,z^n)$ tends to $(b^\infty,z^\infty)$ in $C([0,T]; B^{\frac{1}{p}}_{p,1}\times B^{\frac{1}{p}}_{p,1}).$
		
		Therefore,
		\begin{align}
			\|u_x^n-u_x^{\infty}\|_{L^{\infty}_T (B^{\frac{1}{p}}_{p,1})}\leq \|a^n-a^{\infty}\|_{{L^{\infty}_T (B^{\frac{1}{p}}_{p,1})}}+\|b^n-b^{\infty}\|_{{L^{\infty}_T (B^{\frac{1}{p}}_{p,1})}},\label{ok}
		\end{align}
		\begin{align}
			\|\rho^n-\rho^{\infty}\|_{L^{\infty}_T (B^{\frac{1}{p}}_{p,1})}\leq \|\omega^n-\omega^{\infty}\|_{{L^{\infty}_T (B^{\frac{1}{p}}_{p,1})}}+\|z^n-z^{\infty}\|_{{L^{\infty}_T (B^{\frac{1}{p}}_{p,1})}}.\label{a21}
		\end{align}
		When $n\rightarrow +\infty,$ we get $\eqref{ok}, \eqref{21}\rightarrow 0.$
		That is, 
		$$(u^n_x,\rho^n)\rightarrow(u^{\infty}_x,\rho^\infty)~~ \mbox{in}~~ C([0,T],{B^{\frac1 p }_{p,1}}\times B^{\frac1 p }_{p,1}).$$
		From  \textbf{Step 1} to \textbf{Step 3}, we complete  the proof of Theorem \ref{the1}.
	\end{proof} 
	
	\section{Blow-up}
	\par

In this section, we will present a blow-up result. 

Considering the following initial value problem
\begin{equation}\label{0.2}
	\left\{\begin{aligned}
		&\frac{d}{dt}y(t,\xi)=u(t,-k_3y(t,\xi)),\quad x\in \mathbb{R},\\
		&y(t,\xi)|_{t=0}=\bar{y}(\xi). \\
	\end{aligned} \right. 
\end{equation}
By article \cite{1}, we have the following  three lemmas.
\begin{lemm}\cite{1}
	Let $u\in C([0,T);H^{s,p})\cap  C^1([0,T);H^{s-1,p})$, $s>{\rm max}\{2,\frac{3}{2}+\frac{1}{p}\}$, $p\in(1,\infty)$. Then Eq.\eqref{0.2} has a unique solution $y\in C^1([0,T)\times \mathbb{R};\mathbb{R})$. Futhermore, the map $y(t,\cdot)$ is an increasing diffeomorphism of $\mathbb{R}$ with
	\begin{align}\label{0.3}
			y_x(t,x)=exp(\int^t_0-k_3u_x(s,-k_3y(s,x))dx)>0,~\forall (t,x)\in [0,T)\times \mathbb{R}.
		\end{align} 
\end{lemm}
\begin{lemm}\cite{1}\label{0.18}
	Let $v_0=(\rho_0,u_0)\in H^{s-1,p}\times H^{s,p}$, $s>{\rm max}\{2,\frac{3}{2}+\frac{1}{p}\}$, $p\in(1,\infty)$ and let $T>0$ be the lifespan of the corresponding solution $v$ of Eq.\eqref{eq1}. Then the following equality holds
	\begin{align}\label{0.4}
	\rho(t,-k_3y(t,x))y_x(t,x)=\rho_0(-k_3x),~\forall t\in[0,T).
	\end{align} 
	Moreover, if $k_2=2k_1=4$, $k_3=1$, i.e. $b=2$ in case (\romannumeral1), then
	\begin{align}\label{0.5}
		\frac{d}{dt}\int_{\mathbb{R}}(4\rho^2+u^2+u_x^2)dx=0.
	\end{align} 
		If $k_2=k_1=2$, $k_3=1$, i.e. $b=1$ in case (\romannumeral2), then
	\begin{align}\label{0.6}
		\frac{d}{dt}\int_{\mathbb{R}}(2\rho^2+u^2+u_x^2)dx=0.
	\end{align}
	Let $E_n(t)=\int_{\mathbb{R}}(2n\rho^2+u^2+u_x^2)dx$,$n=1,2$, we have $E_n(t)=E_n(0)$.
\end{lemm}
\begin{lemm}\cite{1}
Assume $v_0=(\rho_0,u_0)\in H^{s-1,p}\times H^{s,p}$, $s>{\rm max}\{2,\frac{3}{2}+\frac{1}{p}\}$, $p\in(1,\infty)$ and let $T>0$ be the lifespan of the corresponding solution $v=(\rho,u)$ of Eq.\eqref{eq1}, then $v$ blows up in finite time if and only if the slope of $u$ satisfies
	\begin{align}\label{0.7}
	 \limsup\limits_{t\uparrow T}\|u_x(t,\cdot)\|_{L^\infty}=+\infty.
	\end{align} 
\end{lemm}
Next, we introduce the following useful lemma.
\begin{lemm}\cite{2}
	Let $T>0$ and $v\in C^1([0,T);H^2)$. Then for all $t\in [0,T) $, there exists at least one point $\xi(t)\in \mathbb{R}$ with
	\begin{align}\label{0.8}
	m(t)\triangleq 	\inf\limits_{x\rightarrow \mathbb{R}}[v_x(t,x)]=v_x(t,\xi(t)).
	\end{align} 
	The function $m(t)$ is absolutely continuous on $(0,T)$ with
		\begin{align}\label{0.9}
		\frac{dm}{dt}=v_{xt}(t,\xi(t))\quad a.e.\quad on (0,T).
	\end{align} 
\end{lemm}
Finally,  we give a blow-up criterions when $k_1=2$.
\begin{theo}
	Let $k_1=2$, $v_0=(\rho_0,u_0)\in H^{s-1,p}\times H^{s,p}$, $s>{\rm max}\{2,\frac{3}{2}+\frac{1}{p}\}$, $p\in(1,\infty)$ and let $T^{\ast}_{v_0}$ be maximal existence time. Assume that there exists $M>0$ such that $\|\rho(t,\cdot)\|_{L^\infty}\leq M$, for $t\in [0,T^{\ast}_{v_0})$ and $\rho_0\in L^1$. If there is a $x_0\in \mathbb{R}$ such that $u_{0x}(x_0)<-\sqrt{2}B$ where $B=\frac{1}{2}(E_n^2(0)+k_2M\|\rho_0\|_{L^1})^\frac{1}{2}$, then $T^{\ast}_{v_0}<\infty$ and the corresponding solution blows up in finite time.
\end{theo}
	\begin{proof}
		Note that $k_1=2$ and $\partial_xP\ast f=P\ast f-f$. Differentiating the second equation of system \eqref{eq1} with respect to $x$, we get 
	\begin{align}\label{0.10}
		u_{tx}=u_x^2+uu_{xx}-(\frac{k_2}{2}\rho^{2}+u^{2}+\frac{1}{2}u_{x}^{2})+P\ast(\frac{k_2}{2}\rho^{2}+u^{2}+\frac{1}{2}u_{x}^{2}).
	\end{align} 
Due to Lemma \ref{0.8}, we have that \begin{align}\label{0.11}
	m(t)\triangleq 	\inf\limits_{x\rightarrow \mathbb{R}}[u_x(t,x)]=u_x(t,\xi(t)).
\end{align} 
So $u_{xx}(t,\xi(t))=0$, for any $t\in [0,T_{v_0}^\ast)$.  By \eqref{0.5} and \eqref{0.6}, one easily has
\begin{align*}
u^2(t,x)&=\int_{-\infty}^{x}uu_xdx-\int^{\infty}_{x}uu_xdx\leq\int_{-\infty}^{+\infty}|uu_x|dx
\leq\frac{1}{2}\int_{\mathbb{R}}(u^2+u_x^2)dx\\
&=\frac{1}{2}(E_n(0)-\int_{\mathbb{R}}2n\rho^2dx)\leq\frac{1}{2}E_n(0),\quad n=1,2.
\end{align*} 
Therefore, 
	\begin{align}\label{0.19}
	\|u\|_{L^\infty}\leq \frac{\sqrt{2}}{2}\|u\|_{H^1}\leq \frac{\sqrt{2}}{2}E_n(0),\quad n=1,2.
\end{align} 
According to Lemma \ref{0.18}, we obtain
$$\|P\ast (\rho^2)\|_{L^\infty}\leq \|P\|_{L^1}\|\rho^2\|_{L^\infty}\leq \frac{1}{2}M\|\rho_0\|_{L^1}.$$
By $P\ast (u^2+\frac{1}{2}u_x^2)\geq \frac{1}{2}u^2$(see \cite{2}), we infer that
	\begin{align}\label{0.20}
	m'(t)&=\frac{1}{2}m^2(t)-uu_{xx}(t,\xi(t))-(\frac{k_2}{2}\rho^{2}+u^{2})(t,\xi(t))+P\ast(\frac{k_2}{2}\rho^{2}+u^{2}+\frac{1}{2}u_{x}^{2})(t,\xi(t))\notag\\
&=\frac{1}{2}m^2(t)-u^2(t,\xi(t))-\frac{k_2}{2}\rho^2(t,\xi(t))+P\ast(u^2+\frac{1}{2}u_x^2)(t,\xi(t))+P\ast(\frac{k_2}{2}\rho^2)(t,\xi(t)).\notag\\
&\geq \frac{1}{2}m^2(t)-B^2,\quad {\rm a.e.~on}\quad [0,T_{v_0}^{\ast}),
\end{align} 	
If $m(0)\geq\sqrt{2}B$, we can easily deduce that
$$m(t)\geq \sqrt{2}B, ~\forall t\in [0,T_{v_0}^{\ast}).$$	
It is obvious that
$$0\leq \frac{2\sqrt{2}B}{m(t)-\sqrt{2}B}\leq e^{-\sqrt{2}Bt}\frac{m(0)+\sqrt{2}B}{m(0)-\sqrt{2}B}-1.$$	
Therefore, there exists $0<T_{0}<-\frac{1}{\sqrt{2}B}\ln\bigg(\frac{m(0)-\sqrt{2}B}{m(0)+\sqrt{2}B}\bigg)$ such that $\lim\limits_{T\rightarrow T_0}m(t)=+\infty.$
	\end{proof}
	
	\noindent\textbf{Acknowledgements.}
This work was partially supported by the National Natural Science Foundation of China (Grant No. 12171493). The authors thank the referees for their valuable comments and suggestions.

	\phantomsection
\addcontentsline{toc}{section}{\refname}
\bibliographystyle{abbrv} 
\bibliography{r}
\end{document}